\documentclass[a4paper,12pt,reqno]{amsart}
\usepackage[T1]{fontenc}
\usepackage[utf8]{inputenc}
\usepackage[libertinus,vvarbb]{newtx}
\usepackage[margin=1in]{geometry}
\usepackage{enumitem}
\usepackage{tikz}
\usepackage{subcaption}
\usepackage{graphicx}
\usepackage[svgnames]{xcolor}
\usepackage{xparse}
\usepackage{xurl}
\usepackage[noadjust]{cite}
\usepackage[bookmarksdepth=2]{hyperref}
\usepackage{comment}
\usepackage{mleftright}
\mleftright

\hypersetup{colorlinks=true,urlcolor=MidnightBlue,linkcolor=MidnightBlue,citecolor=MidnightBlue}

\usepackage[normalem]{ulem}

\newcommand{\stkout}[1]{\ifmmode\text{\sout{\ensuremath{#1}}}\else\sout{#1}\fi}
\newcommand{\crsout}[1]{\ifmmode\text{\xout{\ensuremath{#1}}}\else\xout{#1}\fi}

\newtheorem{theorem}{Theorem}[section]
\newtheorem{thm}{Theorem}[section]
\newtheorem{lemma}[theorem]{Lemma}
\newtheorem{lem}[theorem]{Lemma}
\newtheorem{corollary}[theorem]{Corollary}

\theoremstyle{definition}

\newtheorem{example}[theorem]{Example}
\newtheorem{remark}[theorem]{Remark}

\newcommand{\ind}{\mathbb{1}}
\newcommand{\pr}{\mathbb{P}}
\newcommand{\ex}{\mathbb{E}}

\newcommand{\R}{\mathbb{R}}

\newcommand{\ph}{\varphi}
\newcommand{\eps}{\varepsilon}
\newcommand{\thet}{\vartheta}

\newcommand{\lv}{\lvert}
\newcommand{\rv}{\rvert}
\newcommand{\lV}{\lVert}
\newcommand{\rV}{\rVert}

\renewcommand{\le}{\leqslant}
\renewcommand{\ge}{\geqslant}

\NewDocumentCommand{\formula}{ssom}{%
 \IfBooleanTF{#1}{%
  \IfBooleanTF{#2}{%
   \IfValueTF{#3}%
    {\begin{align}\label{#3}\begin{gathered}#4\end{gathered}\end{align}}%
    {\begin{gather}#4\end{gather}}%
  }{%
   \IfValueTF{#3}%
    {\begin{align}\label{#3}\begin{aligned}#4\end{aligned}\end{align}}%
    {\begin{gather*}#4\end{gather*}}%
  }%
 }{%
  \IfValueTF{#3}%
   {\begin{align}\label{#3}#4\end{align}}%
   {\begin{align*}#4\end{align*}}%
 }%
}

\newcommand{\form}{\mathscr{E}}
\DeclareMathOperator{\dom}{dom}

\begin{document}

\title{Large time behavior of L\'evy processes and their nonlocal Schrödinger semigroups}
\author{Mateusz Kwaśnicki {$^\dag$}, Phanuel Mariano {$^\ddag$}, Hugo Panzo, Jing Wang {$^\S$}}
\thanks{$\dag$ Work supported by the National Science Centre, Poland, grant no.\@ 2023/49/B/ST1/04303}
\thanks{$\ddag$ Work supported by the National Science Foundation, USA, grant no. DMS-2316968.}
\thanks{$\S$ Work supported by the National Science Foundation, USA, grant no. DMS-2246817.}
\address{Mateusz Kwaśnicki \\ \textnormal{Department of Analysis and Stochastic Processes \\ Wrocław University of Science and Technology \\ Wybrzeże Wyspiańskiego 27 \\ 50-370 Wrocław, Poland}}
\email{\href{mailto:mateusz.kwasnicki@pwr.edu.pl}{\textsf{mateusz.kwasnicki@pwr.edu.pl}}}
\address{Phanuel Mariano \\ \textnormal{Department of Mathematics\\ Union College\\ 807 Union Street \\ Schenectady, NY 12308}}
\email{\href{mailto:marianop@union.edu}{\textsf{marianop@union.edu}}}
\address{Hugo Panzo \\ \textnormal{Department of Mathematics and Statistics \\ Saint Louis University \\ St.~Louis, MO 63103, USA}}
\email{\href{mailto:hugo.panzo@slu.edu}{\textsf{hugo.panzo@slu.edu}}}
\address{Jing Wang \\ \textnormal{Department of Mathematics, Department of Statistics \\ Purdue University, West Lafayette, IN 47907, USA}}
\email{\href{mailto:jingwang@purdue.edu}{\textsf{jingwang@purdue.edu}}}

\begin{abstract}
We study the large time asymptotics of the Feynman--Kac semigroups of the symmetric Lévy process on unbounded open sets. Our main result proves the exact exponential asymptotic decay rate for the survival probability given in terms of the bottom of the spectrum of the associated nonlocal Schr\"odinger operator. We also prove quantitative upper bounds with an explicit polynomial correction. The proof is probabilistic and is done by decomposing the Lévy process into a finite-range jump process collecting the small jumps and an independent compound Poisson process describing the large jumps. Our approach gives a direct link between the spectral properties of nonlocal Schrödinger operators and pointwise decay of survival probabilities, which was previously unknown for jump processes. 

\end{abstract}

\maketitle

%
%

\section{Introduction}

We consider a symmetric Lévy process $X_{t}$ in $\mathbb{R}^{d}$
with a covariance matrix $a=\left(a_{ij}:i,j=1,\dots,d\right)$ of
the Brownian component and with a symmetric Lévy measure $\nu$ on
$\mathbb{R}^{d}\setminus\left\{ 0\right\} $ which describes the
jumps of $X_{t}$ . Its generator is given by 
\[
\mathcal{L}f\left(x\right)=\frac{1}{2}\sum_{i,j=1}^{d}a_{ij}\frac{\partial^{2}}{\partial x_{i}\partial x_{j}}f\left(x\right)+\text{P.V.}\int_{\mathbb{R}^{d}\setminus\left\{ 0\right\} }\left(f\left(x+y\right)-f\left(x\right)\right)\nu\left(dy\right).
\]
When $\alpha\in\left(0,2\right)$,  
\[
\nu\left(dy\right)=c_{d,\alpha}\left|y\right|^{-d-\alpha}dy,\quad c_{d,\alpha}=2^{\alpha}\pi^{-d/2}\Gamma(\tfrac{d+\alpha}{2})|\Gamma(-\tfrac{\alpha}{2})|^{-1}
\]
and $a_{i,j}=0$ for $i,j=1,2,\dots,d$, the process $X_{t}$ is the isotropic
$\alpha$-stable Lévy process. Its generator is
the Fractional Laplacian $\mathcal{L}=-\left(-\Delta\right)^{\alpha/2}$.
When $\nu=0$ and $a_{ij}=\delta_{ij}$, the process $X_t$ is the Brownian
motion with generator $\mathcal{L}=\frac{1}{2}\Delta$.  

Let $D\subset\R^d$ be a connected domain, and let
\begin{align*}\label{eq-tau-D}
    \tau_{D}:=\inf\left\{ t>0:X_{t}\notin D\right\} .
\end{align*}
be the first exit time of $X_t$ from $D$. For $x\in D$, the distribution function $\mathbb{P}_{x}\left(\tau_{D}>t\right)$ of $\tau_D$ records the survival probability of $X_t$ in $D$. The large time behavior of this survival probability has been widely studied. When $D$ is bounded, the generator of $X_t$ killed upon exiting $D$ has a discrete spectrum on $L^{2}\left(D\right)$. Using the spectral expansion of the killed heat kernel and some mild assumptions, one obtains the following asymptotic decay
\[
\lim_{t\to\infty}\frac{1}{t}\log\mathbb{P}^{x}\left(\tau_{D}>t\right)=-\lambda_{1}\left(D\right),
\]
where $\lambda_{1}\left(D\right)$ is the first Dirichlet eigenvalue of the killed process in $D$. This estimate is classical for the Brownian motion and obtained for $\alpha$-stable processes in \cite{Chen-Song-1997,Kulczycki-1998}.

This spectral-expansion argument does not apply to unbounded domains, since the Dirichlet generator may not have purely discrete $L^2\left(D\right)$ spectrum. Nevertheless, when an unbounded domain has a positive bottom of the Dirichlet Laplacian $-\frac{1}{2}\left.\Delta\right\vert_D$, given by 
\[
\lambda_D:=\inf \biggl\{ \int_{D}\frac{1}{2}\left|\nabla f\right|^{2}dx : f\in C_{c}\left(D\right) , \lVert f \rVert_2 = 1 \biggr\},
\]  
one still expects the survival probability to decay exponential at the rate of $\lambda_D$. Such an asymptotic estimate is known for local processes. For instance, B. Simon  (\cite{Simon-1981,Simon-1981b}) proved that for the Brownian motion,
\begin{equation}
\lim_{t\to\infty}\frac{\log\mathbb{P}^{x}\left(\tau_{D}>t\right)}{t}=-\lambda_D, \quad x\in D\label{eq:coj-main}
\end{equation}
for any unbounded domain $D\subset \mathbb{R}^d$.  See also \cite[Theorem 3.1.2]{Sznitman-1998} for a proof of \eqref{eq:coj-main}. 
  Many related results for local processes have since been established, including in the works \cite{Vogt-2019a,Ouhabaz-2006,Sznitman-1998,Bakharev-Nazarov-JFA-2021,Lifshits-Nazarov-2019}. More recently, \cite{Mariano-Wang-2022} extended such estimates to general Hunt processes satisfying sub-Gaussian heat kernel bounds in unbounded domains on metric measure spaces.

Simon’s result is in fact formulated in a more general setting, namely for the Feynman--Kac semigroup associated with the Schr{\"o}dinger operator
$
-\frac{1}{2}\Delta+V,
$
under suitable assumptions on the potential $V$. For a domain $D\subset\mathbb{R}^{d}$,
this semigroup is given by
\[
P_{t}^{D,V}f\left(x\right)=\mathbb{E}_{x}\left[f\left(X_{t}\right)\exp\left(-\int_{0}^{t}V\left(X_{s}\right)ds\right)1_{\left\{ \tau_{D}>t\right\} }\right].
\]
Simon proved that for any $x\in D$, 
\[
\lim_{t\to\infty}\frac{\log P_{t}^{D,V}\ind_{D}\left(x\right)}{t}=\lim_{t\to\infty}\frac{\log{}\left\Vert P_{t}^{D,V}\right\Vert _{\infty\to \infty}}{t}=-\lambda_{D,V},
\]
where
\[
\lambda_{D,V}:=\inf \biggl\{ \int_{D}\frac{1}{2}\left|\nabla f\right|^{2}dx + \int_D \lv f(x) \rv^2 V(x) dx : f\in C_{c}\left(D\right) , \lVert f \rVert_2 = 1 \biggr\}.
\]

However, a gap remains in the literature concerning the sharp limiting behavior of the survival probability of non-local jump processes in unbounded domains. In particular we would like to consider a  symmetric Lévy process $X_t$  in $\R^d$ with Dirichlet form $\form$ in an unbounded domain 
$D\subset \mathbb{R}^d$ with  $\lambda_D>0$, where 
\formula{
 \lambda_D & = \inf \biggl\{ \form(f) : f \in C_c^\infty(D) , \, \lV f \rV_2 = 1 \biggr \} .
}
Although it is natural to expect that \eqref{eq:coj-main} remains valid in this setting,  to the authors' knowledge, a proof does not seem to
be available in the literature, even for isotropic $\alpha$-stable processes
on $\mathbb{R}^d$. Several works have studied related questions, including large-time
asymptotics, heat kernel estimates, and bounds on exit times for non-local
processes; see, among others, \cite{Banuelos-Mariano-Wang-2024,Chen-Kim-Song-2010,Kim-Song-2014,Siudeja-2006, Chen-Kim-Wang-2022,Banuelos-Latala-Mendez-2001,Bogdan-Grzywny-2010,Barrier-Bogdan-Grzywny-Ryznar-2015,Kwasnicki-2009,Giorgi-Smits-2010,Panzo2021,Chen-Kim-Wang-2019,KaletaLorinczi2015,ChenWang2016}. However, many of these results either rely on bounded domains or impose regularity assumptions on the boundary. 

In \cite{Mariano-Wang-2025}, the authors obtained the asymptotic bounds
\begin{equation}\label{Mariano-WangBound}
-\lambda_D\leq\liminf_{t\to\infty}\frac{\log\mathbb{P}_{x}\left(\tau_{D}>t\right)}{t}\leq\limsup_{t\to\infty}\frac{\log\mathbb{P}_{x}\left(\tau_{D}>t\right)}{t}\leq-\frac{\lambda_D}{1+\frac{d_{1}}{4\alpha_{1}}},    
\end{equation}
for symmetric jump processes in unbounded domains  with $\lambda_D>0$ on metric measure spaces. The proof relies on heat kernel estimates and certain iteration techniques. 
The constants $d_1$ and $\alpha_1$ come from the volume growth and scaling assumptions on the underlying space and process. In the case of an isotropic $\alpha$-stable process on $\mathbb{R}^d$, one has $d_1=d$ and $\alpha_1=\alpha$. In a related work, Chen and Yu \cite{Chen-Yu-2026} obtained the upper bound
$\frac{-\lambda_D}{1+d_1 / (2 \alpha_1)}$
for the limsup in \eqref{Mariano-WangBound}. This estimate applies to a more general class of processes, but is weaker than that in \eqref{Mariano-WangBound}. Thus, while these results give useful upper bounds, they do not identify the exact exponential decay rate in \eqref{eq:coj-main}. To the best of our knowledge, the only previously known unbounded
domains with continuous spectrum and  $\lambda_D>0$ for which the
sharp asymptotic was established are the increasing horn-shaped domains
considered in \cite[Theorem 1.5]{Mariano-Wang-2025}.
The main difficulty comes from the polynomial tails of  non-local processes, which prevent standard localization methods from obtaining matching upper and lower bounds. The exact large time asymptotics of jump processes in unbounded
domains with  $\lambda_D=0$ have been widely studied, for example in \cite{Siudeja-2006,Mendez-Hernandez-2002,Mendez-Hernandez-2007,Banuelos-Bogdan-2004,Banuelos-Bogdan-2005,Kwasnicki-2009,Chen-Tokle-2011,BogdanPalmowskiWang2018}. 

In this paper, we develop a new method that allows us to prove the sharp exponential decay rate of the survival probability of a Lévy process in $\R^d$. It reveals the connection  between the spectral properties of non-local Schrödinger operators and the pointwise large-time asymptotics of survival probabilities of jump processes irreducible in $D$. We say that the killed process $X_t^D$  is \emph{irreducible} in $D$ if for every $x \in D$ and every open set $A \subseteq D$ with nonempty interior, there exists $t>0$ such that
\[
\pr^x(X_t \in A, \, t < \tau_D) > 0
\]
The main result is as follows. 

\begin{theorem}
\label{thm:general}
Let $X_t$ be a symmetric Lévy process in $\R^d$ with Dirichlet form $\form$, and assume that for $t$ large enough, $X_t$ has a square integrable density function. Let $D \subseteq \R^d$ be a nonempty open set, and $\tau_D$ the first exit time of $X_t$ from $D$. Denote by $\lambda_D$ the bottom of the spectrum of $X_t$ on $D$:
\formula{
 \lambda_D & = \inf \biggl\{ \form(f) : f \in C_c^\infty(D) , \, \lV f \rV_2 = 1 \biggr \} .
}
Then for every $x \in D$ we have
\formula{
 \limsup_{t \to \infty} \frac{\log \pr^x(\tau_D > t)}{t} & \le -\lambda_D.
}
If in addition, the killed process is irreducible in $D$, 
then
\formula{
 \lim_{t \to \infty} \frac{\log \pr^x(\tau_D > t)}{t} & = -\lambda_D .
}
More generally, let $V$ be a nonnegative, locally bounded Borel function on $D$, and
\formula{
 \lambda_{D, V} & = \inf \biggl\{ \form(f) + \int_D \lv f(x) \rv^2 V(x) dx : f \in C_c^\infty(D) , \, \lV f \rV_2 = 1 \biggr \}
}
be the bottom of the spectrum of the Schrödinger operator with potential $V$. Then for every $x \in D$ we have
\formula{
 \limsup_{t \to \infty} \frac{1}{t} \log \biggl( \ex^x \biggl[ \exp \biggl( -\int_0^t V(X_s) ds \biggr) \ind_{\{\tau_D > t\}} \biggr] \biggr) & \le -\lambda_{D, V} ,
}
If the killed process in $D$ is irreducible, then
\formula{
 \lim_{t \to \infty} \frac{1}{t} \log \biggl( \ex^x \biggl[ \exp \biggl( -\int_0^t V(X_s) ds \biggr) \ind_{\{\tau_D > t\}} \biggr] \biggr) & = -\lambda_{D, V} ,
}
\end{theorem}
\begin{remark}
    Note that by shifting, our proof applies to locally bounded $V$ that is bounded from below.
\end{remark}

The main idea for the proof of Theorem \ref{thm:general} is to decompose $X_t$ into the sum of two independent Lévy processes, $Y_t$ and $Z_t$, where $Y_t$ is a Lévy process with jumps not exceeding $r_0$, and $Z_t$ is a compound Poisson process that collects all jumps of $X_t$ larger than $r_0$. One can then rewrite the survival probability for $X_t$ in $D$ as the survival probability for $Y_t$ in a time-dependent domain $D - Z_t$. By taking away the large jumps, we can show that the truncated Lévy  process $Y_t$ decays super-exponentially fast in the proper regime. This gives us the fast enough decay to analyse the survival probability and show it has the right exponential decay. 

Connecting $\lambda_D$ and $\lambda_{D, V}$ to the large-time survival probability in Theorem \ref{thm:general} opens up many applications for jump processes. For instance, the seminal results by Chen, Kim and Song in \cite{Chen-Kim-Song-2010} established heat kernel estimates for $\alpha$-stable processes on $C^{1,1}$ open sets. Their estimates cover both small and large times, but in the large-time regime the domain $D$ is assumed to be bounded. In that case, they prove that, for $t>T$, $p_D\left(t,x,y\right)\asymp e^{-\lambda_1 t} \delta_D(x)^{\alpha/2}\delta_D(y)^{\alpha/2}$ for $t>T$ where $\lambda_1$ is the first Dirichlet eigenvalue and $\delta_D(x)=\text{dist}(x,D^c)$. For unbounded domains, one should expect that large-time heat kernel estimates to be governed by the bottom of the spectrum with a form of $p_D\left(t,x,y\right)\asymp e^{-\lambda_D t} K(t,x,y)$ for $t>T$ for some $K(t,x,y)$. 

One immediate application of Theorem  \ref{thm:general} is the $L^p$ independence of the non-local Feynman--Kac semigroup associated to the Schr\"odinger operators for Lévy processes.

\begin{corollary}[$L^p$-independence]
Consider $X_t$ as in Theorem \ref{thm:general} and assume that $D$ is irreducible for $X_t$.  If $P_t^D$ is the Feynman--Kac
semigroup given by 
\[
P_{t}^{D,V}f\left(x\right)=\mathbb{E}^{x}\left[f\left(X_{t}\right)\exp\left(-\int_{0}^{t}V\left(X_{s}\right)ds\right)1_{\left\{ \tau_{D}>t\right\} }\right],
\]
then, for any $p\in\left[1,\infty\right]$, we have
\[
\lim_{t\to\infty}\frac{1}{t}\log\left\Vert P_{t}^{D,V}\right\Vert _{p\to p}=\lim_{t\to\infty}\frac{1}{t}\log\left\Vert P_{t}^{D,V}\right\Vert _{2\to2}=\lambda_{D,V}.
\]
\end{corollary}

The proof follows by the same interpolation arguments given in \cite{Simon-1981,Simon-1981b}. We also note the various related works on the $L^p$ independence for non-local semigroups in \cite{ChenKimKuwae2019,TakedaWada2016,Tawara2010,GuillinLuNectouxWu2026}.

Simon’s original result for diffusions not only gave an asymptotic limit, but also a stronger uniform bound on $P_t^{D,V}\ind_D(x)$, valid for all $x\in D$ and $t>0$. In particular he obtained the following, now classical, upper bound in \cite[Eq.  1.9]{Simon-1981}, with suitable assumptions on the potential $V$,
\begin{equation}\label{eq:simon:d/2}
P_{t}^{D,V}\ind_D\left(x\right)\leq C\left(1+t\right)^{d/2}e^{-\lambda_{D,V} t},\,\,\,t>0,x\in D.
\end{equation}
The large time behavior of $P_{t}^D$ for diffusions have been widely
studied. In particular Simon and Davies went on to prove several results for
different classes of potentials \cite{Simon-1981,Davies-Simon-1991,Simon-1981b}. Ouhabaz in \cite[Theorem 4]{Ouhabaz-2006} was able to improve \eqref{eq:simon:d/2}
to 
\begin{equation}
P_{t}^{D,V}\ind_D\left(x\right)\leq C\left(1+t\log t\right)^{d/4}e^{-\lambda_{D,V} t},\,\,\,t>0,x\in D.\label{eq:Ouhabaz}
\end{equation}
The current sharpest result for local processes is given by Vogt in
\cite[Theorem 1.1]{Vogt-2019a} where he removes the $\log$ term and obtains the following,
\begin{equation}
P_{t}^{D,V}\ind_D\left(x\right)\leq C\left(1+t\right)^{d/4}e^{-\lambda_{D,V} t},x\in D.\label{eq:Vogt}
\end{equation}
It is remarked by Vogt \cite[Remark 1.2]{Vogt-2019a} that the exponent $\frac{d}{4}$ is sharp
in dimension $d=4$. This is because Simon in \cite[Theorem 4.1]{Simon-1981b} gave examples of potentials $V$ in $\R^4$ where $\left\Vert P^{D,V}_{t}\right\Vert _{\infty\to\infty}$
grows like $\frac{t}{\log t}$. Up to now, it was still unknown
what the sharp exponent is for other dimensions (See Section \ref{sec:AI}).

Our techniques also allow us prove the following sharp exponential decay bound with  an explicit polynomial correction for all $t>0$ and
$x\in D$ for the isotropic stable Lévy process with index $\alpha \in (0, 2)$. 
\begin{theorem}
\label{thm:stable}
Suppose that $X_t$ is the isotropic stable Lévy process with index $\alpha \in (0, 2)$. Let $D \subseteq \R^d$ be a nonempty open set. There exists a constant $C>0$ such that for every $x \in D$ we have
\begin{equation}\label{MainEstimate-1}
 P_{t}^{D}\ind_D\left(x\right)=\pr^x(\tau_D > t)  \le C\left(1+t\right)^{\frac{d}{2\alpha}+\frac{d}{2}}e^{-\lambda_D t}.
\end{equation}
More generally, if $V$ is a nonnegative, locally bounded Borel function on $D$, there exists a constant $C>0$ such that for every $x\in D$,
\begin{equation}\label{MainEstimate-2}
P_{t}^{D,V}\ind_D\left(x\right)= \ex^x \biggl[ \exp \biggl( -\int_0^t V(X_s) ds \biggr) \ind_{\{\tau_D > t\}} \biggr] \leq   C\left(1+t\right)^{\frac{d}{2\alpha}+\frac{d}{2}}e^{-\lambda_{D,V} t}.
\end{equation}
\end{theorem}

We point out that the exponent in the polynomial term in \eqref{MainEstimate-1} and \eqref{MainEstimate-2} is not sharp. One should be able to improve the bound in \eqref{MainEstimate-1} and \eqref{MainEstimate-2} to $C\left(1+t\right)^{\frac{d}{2\alpha}}e^{-\lambda_{D,V} t}$ with the improved exponent of $\frac{d}{2\alpha}$.   The reason is analogous to the diffusion case. If $\alpha=2$, then this would match Vogt's result \eqref{eq:Vogt}. And similar to Simon's result in \cite[Theorem 4.1]{Simon-1981b}, it was shown in \cite[Theorem 1.1]{TakedaWada2016} examples of potentials $V$  for the symmetric $\alpha$-stable process where $\left\Vert P^{D,V}_{t}\right\Vert _{\infty\to\infty}$
grows like $\frac{t}{\log t}$ if $d = 2 \alpha$. In fact, this conjecture can be proved on a metric measure space with the assistance of a large language model, ChatGPT Astra. 

\subsection{Toward the sharp polynomial exponent}\label{sec:AI}
The results of this paper were obtained by the authors without the use of AI assisted tools. 
To clearly distinguish the human contribution from the AI assisted work, we state here the sharper result obtained with AI assistance in the setting of $\R^d$ and defer its proof to the forthcoming paper \cite{MetricMeasureSurvival-2026}.

\begin{theorem}[\cite{MetricMeasureSurvival-2026}]
\label{thm:stable-sharp}
Suppose that $X_t$ is the isotropic $\alpha$-stable Lévy process in $\R^d$, with $\alpha\in(0,2]$, and let $M_t$ be the Feynman--Kac multiplicative functional
\[
M_t=
\begin{cases}
\displaystyle
\exp\left(-\int_0^t V(X_s)\,ds\right),& t<\tau_D,\\[1ex]
0,& t\ge \tau_D,
\end{cases}
\]
where $D\subseteq\R^d$ is a nonempty open set and
$V:D\to[0,\infty)$ is locally bounded and Borel measurable.

Then there exists a constant $C=C(d,\alpha)>0$, independent of
$D$ and $V$, such that, for every $t>0$ and $x\in D$,
\begin{equation}
\label{eq:stable-sharp}
P_t^M\ind(x)
=
\ex^x\left[
\exp\left(-\int_0^t V(X_s)\,ds\right)
\ind_{\{\tau_D>t\}}
\right]
\le
C(1+\lambda_{D,V} t)^{\frac{d}{2\alpha}}
e^{-\lambda_{D,V} t}.
\end{equation}
\end{theorem}

The following result shows the result above is actually sharp. 

\begin{theorem}[\cite{MetricMeasureSurvival-2026}]
\label{thm:stable-optimality}

Let $d\ge1$, $0<\alpha\le2$, and $0<\beta<d/2$. Set
\[
 h(x)=(1+|x|^2)^{-\beta/2},\qquad
 W(x)=-\frac{(-\Delta)^{\alpha/2}h(x)}{h(x)}.
\]
Then $W$ is bounded and continuous, with $W(x)\to0$ as $|x|\to\infty$.
For any $\lambda>\|W_-\|_\infty$, the bounded potential
$V=\lambda+W$ is nonnegative, and $(-\Delta)^{\alpha/2}+V$ on $\R^d$ satisfies
$\lambda_{\R^d,V}=\lambda$. Moreover, for some $c>0$,
\begin{equation}
\label{eq:stable-lower-example}
P_t^{\R^d,V}\ind(0)
\ge
c\,e^{-\lambda_{\R^d,V} t}t^{\beta/\alpha},
\qquad t\ge1.
\end{equation}

Consequently, there is no exponent
\[
\gamma<\frac{d}{2\alpha}
\]
for which an estimate of the form
\[
P_t^{D,V}\ind(x)
\le
C_{D,V}(1+t)^\gamma e^{-\lambda_{D,V} t}
\]
holds for every nonnegative bounded potential $V$ and domain $D\subset \R^d$, even when the
constant $C_{D,V}$ is allowed to depend on the potential.
Therefore the exponent $d/(2\alpha)$ in
Theorem~\ref{thm:stable-sharp} is optimal for general potential
killing.
\end{theorem}

We note that the idea of the construction of this potential goes back to the work of Kaleta and L\H{o}rinczi \cite[(1.1)]{KaletaLorinczi2020} and the computations in \cite[Corollary 1]{DydaKuznetsovKwasnicki2017}.

\subsection{Organization of the paper:}

We organize the paper as follows. In Section \ref{sec:prelim} we give the necessary preliminaries for our problem. In Section \ref{sec:gen} we prove Theorem \ref{thm:general} for general Lévy processes. In Section \ref{sec:stable} we prove the sharp upper bounds for the isotropic $\alpha$-stable process in Theorem \ref{thm:stable}. In Section \ref{sec:examples} we recall some examples. 

%
%

\section{Preliminaries}\label{sec:prelim}


\subsection{Lévy processes}

Throughout this paper, $X_t$ is a symmetric Lévy process in $\R^d$. This process is described by the covariance matrix $a = (a_{ij} : i, j = 1, 2, \ldots, d)$ of the Brownian component and the symmetric Lévy measure $\nu$ on $\R^d \setminus \{0\}$, which describes the intensity of jumps of $X_t$. Recall that $a$ can be an arbitrary nonnegative definite matrix, and $\nu$ can be an arbitrary nonnegative measure on $\R^d \setminus \{0\}$ such that
\formula{
 \int_{\R^d \setminus \{0\}} \min\{1, \lv x \rv^2\} \nu(dx) & < \infty \qquad \text{and} \qquad \text{$\nu(-A) = \nu(A)$ for every Borel $A \subseteq \R^d \setminus \{0\}$.}
}
As it is customary in the theory of Markov processes, we consider the process $X_t$ started at an arbitrary point $x \in \R^d$, and we denote by $\pr^x$ and $\ex^x$ the corresponding probability and expectation.

The transition kernel of $X_t$ is given by
\formula{
 P_t(x, A) & = \pr^x \bigl[ X_t \in A \bigr] .
}
We write $P_t(A)$ instead of $P_t(0, A)$, and since $X_t$ is space-homogeneous, we have $P_t(x, A) = P_t(A - x)$. Whenever the integral makes sense, we write
\formula{
 P_t f(x) & = \int_{\R^d} f(y) P_t(x, dy) = \ex^x \bigl[ f(X_t) \bigr] .
}

When $\alpha \in (0, 2)$, we say that $X_t$ is the isotropic $\alpha$-stable Lévy process in $\R^d$ if
\formula{
 \nu(x) & = c_{d, \alpha} |x|^{-d - \alpha} ,
}
where $c_{d, \alpha} = 2^\alpha \pi^{-d / 2} \Gamma(\tfrac{d + \alpha}{2}) \lv \Gamma(-\tfrac{\alpha}{2}) \rv^{-1}$, and $a_{ij} = 0$ for $i, j = 1, 2, \ldots, d$.


\subsection{Killed processes}

Let $D$ be a nonempty open subset of $\R^d$. By
\formula{
 \tau_D & = \inf\{t \ge 0 : X_t \notin D\}
}
we denote the first exit time of $X_t$ from $D$. We introduce the killed process
\formula{
 X_t^D & = \begin{cases} X_t & \text{if $t < \tau_D$;} \\ \partial & \text{if $t \ge \tau_D$,} \end{cases}
}
where $\partial$ is an auxiliary cemetery state. The corresponding transition kernel is defined by
\formula{
 P_t^D(x, A) & = \pr^x \bigl[ X_t^D \in A \bigr] = \pr^x \bigl[ X_t \in A ; \, t < \tau_D \bigr] ,
}
and we write
\formula{
 P_t^D f(x) & = \ex^x \bigl[ f(X_t^D) \bigr] = \ex^x \bigl[ f(X_t) \ind_{\{\tau_D > t\}} \bigr]
}
whenever the integral makes sense. Here we adopt the convention that $f(\partial) = 0$ for every function $f$ on $\R^d$. Note that if $x \notin D$, then $X_t^D = \partial$ for all $t \ge 0$ almost surely with respect to $\pr^x$, and $P_t^D f(x) = 0$ for all $t \ge 0$.


\subsection{Multiplicative functionals}

Let $D$ be a nonempty open subset of $\R^d$, and $V$ a locally bounded nonnegative function on $D$. We consider the Feynman--Kac functional $M_t$, defined by
\formula[eq:feynman:kac]{
 M_t & = \begin{cases} \displaystyle \exp\biggl(-\int_0^t V(X_s) ds\biggr) & \text{if $t < \tau_D$;} \\ 0 & \text{if $t \ge \tau_D$.} \end{cases}
}
Intuitively, the value of $M_t$ describes either the diminishing mass of a particle, or the (conditional) probability of that the particle is alive at a given time.

More generally, we consider an arbitrary nonincreasing multiplicative functional $M_t$ which is continuous execpt a possible jumps to $0$, and which starts from either $1$ or $0$. In this case we denote by $D$ the set of permanent points of $M_t$: for $x \in D$, we have $M_t > 0$ for sufficiently small $t > 0$, while for $x \notin D$, we have $M_t = 0$ for all $t \ge 0$ almost surely with respect to $\pr^x$. The functional $M_t$ admits a description similar to~\eqref{eq:feynman:kac}, but with $V(x)$ replaced by a suitable nonnegative measure $V(dx)$, namely, the Revuz measure of $-\log M_t$. We refer to Chapter~III in~\cite{Blumenthal-Getoor} and Sections~5.1 and~6.1 in~\cite{Fukushima-book-ed1-1994} for further discussion.

We remark that the pair $(X_t, M_t)$ corresponds to a suitable part of the process $X_t$, and $M_t$ describes the gradual killing of $X_t$, but we do not need this notion. Our key object of interest is, however, the corresponding transition kernel, defined by
\formula{
 P_t^M(x, A) & = \ex^x \bigl[ M_t \ind_{\{X_t \in A\}} \bigr] .
}
We write
\formula{
 P_t^M f(x) & = \ex^x \bigl[ f(X_t) M_t \bigr]
}
whenever the integral makes sense. In particular, $P_t^M f(x) = 0$ when $x \notin D$ and $t \ge 0$.

Observe that choosing $V$ to be constant zero in $D$, we obtain
\formula{
 M_t & = \begin{cases} 1 & \text{if $t < \tau_D$;} \\ 0 & \text{if $t \ge \tau_D$,} \end{cases}
}
and in this case $P_t^M = P_t^D$.


\subsection{Dirichlet forms}

We denote the Dirichlet form of $X_t$ by
\formula{
 \form(f) & = \frac{1}{2} \int_{\R^d} \sum_{i, j = 1}^d a_{ij} \frac{\partial f}{\partial x_i}(x) \frac{\partial f}{\partial x_j}(x) dx + \frac{1}{2} \int_{\R^d} \int_{\R^d \setminus \{0\}} (f(x + y) - f(x))^2 \nu(dy) dx ,
}
defined initially for $f \in C_c^\infty(\R^d)$. We also let
\formula[eq:form]{
 \form_M(f) & = \form(f) + \int_{\R^d} \lv f(x) \rv^2 V(dx) ,
}
where $V$ is the Revuz measure of $-\log M_t$. When $M_t$ is the Feynman--Kac functional given by~\eqref{eq:feynman:kac}, we set $V(dx) = V(x) dx$.

The domain of $\form$, denoted $\dom(\form)$, is the completion of $C_c^\infty(\R^d)$ with respect to the norm $(\lV f \rV_2^2 + \form(f))^{1 / 2}$. Clearly, $\dom(\form)$ is a subset of $L^2(\R^d)$. The Dirichlet form $\form_D$ of the killed process $X_t^D$ is the restriction of $\form$ to the completion of $C_c^\infty(D)$ with respect to the same norm. Finally, the domain of $\form_M$, denoted $\dom(\form_M)$, is the completion of $C_c^\infty(D)$ with respect to the norm $(\lV f \rV_2^2 + \form_M(f))^{1 / 2}$. Alternatively, $f \in \dom(\form_M)$ if and only if $f \in \dom(\form_D)$ and the integral in~\eqref{eq:form} is finite (with an appropriate quasi-continuous choice of the representative of $f$ in case the measure $V$ is not absolutely continuous; see Section~6.1 in~\cite{Fukushima-book-ed1-1994}).

The bottom of the spectrum of $X_t^D$ is given by the Rayleigh--Ritz variational formula
\formula{
 \lambda_D & = \inf \biggl\{ \form(f) : f \in \dom(\form_D), \lV f \rV_2 = 1 \biggr\} .
}
Similarly, the bottom of the spectrum of $(X_t, M_t)$ is
\formula{
 \lambda_M & = \inf \biggl\{ \form_M(f) : f \in \dom(\form_M) , \lV f \rV_2 = 1 \biggr\} .
}
This was denoted by $\lambda_{D, V}$ in the Introduction, but, for simplicity, from now on we will use the shorter symbol $\lambda_M$.

For every $f \in L^2(D)$, we have
\formula{
 \lV P_t^M f \rV_2 & \le \exp(-\lambda_M t) \lV f \rV_2 .
}
When $P_t^M$ is compact (so for example when $D$ is bounded), then there is a complete orthonormal system of eigenfunctions of $P_t^M$ (the same one for all $t \ge 0$), and the corresponding eigenfunctions $\ph_n$ satisfy $P_t^M \ph_n(x) = \exp(-\lambda_n t) \ph_n(x)$ for $x \in D$ and $t \ge 0$, where $\lambda_M = \lambda_1 \le \lambda_2 \le \cdots$.


\subsection{Assumptions}

Throughout this work, we assume that $(X_t, M_t)$ is irreducible, in the sense that
\formula[eq:irreducible]{
 P_t^M(x, A) & > 0 \qquad \text{when $x \in D$, $t > 0$ and $A \subset D$ has nonempty interior.}
}
We additionally assume eventual ultracontractivity of the transition operators of $X_t$:
\formula[eq:ultracontractive]{
 & \text{$P_t f$ is bounded when $f \in L^2(D)$ and $t$ is large enough,}
}
and we denote
\formula{
 \lV P_t \rV_{2, \infty} & = \sup_f \frac{\lV P_t f \rV_\infty}{\lV f \rV_2} ,
}
with the supremum over all nonzero $f \in L^2(D)$.


\subsection{Notation}

We occasionally write $X(t)$, $M(t)$, $\tau(D)$, $\lambda(D)$ or $\lambda(M)$ to improve clarity or avoid nested indices. We also denote $s \wedge t = \min\{s, t\}$.

%
%

\section{General Lévy processes}\label{sec:gen}

Our goal in this section is to prove Theorem~\ref{thm:general}: for every $x \in D$,
\formula[eq:goal]{
 \lim_{t \to \infty} \frac{\log P_t^M \ind(x)}{t} & = -\lambda_M .
}
We begin with a relatively simple proof of the lower bound for the limit in Section~\ref{sec:lower}. The proof of the upper bound is much more involved. In Sections~\ref{sec:bennett} and~\ref{sec:varying}, we prove two auxiliary results: a simple large deviations estimate and an $L^2$ bound for time-inhomogeneous multiplicative functionals. The actual proof begins in Section~\ref{sec:decomposition} with a decomposition of $X_t$ into two independent Lévy processes $\tilde X_t$ and $Z_t$, which gather small and large jumps of $X_t$, respectively. With the path of $Z_t$ fixed, we view $M_t$ as a time-inhomogeneous multiplicative functional of $\tilde X_t$, and use our auxiliary estimate to get one ingredient of the upper bound in~\eqref{eq:goal}. The other ingredient is found in Section~\ref{sec:upper} using the large deviations estimate.


\subsection{Lower bound}
\label{sec:lower}

As mentioned above, the lower bound in~\eqref{eq:goal} is easy.

\begin{lemma}
\label{lem:lower}
If~\eqref{eq:irreducible} and~\eqref{eq:ultracontractive} hold true, then, for every $x \in D$,
\formula[eq:lower]{
 \liminf_{t \to \infty} \frac{\log P_t^M \ind(x)}{t} & \ge -\lambda_M .
}
\end{lemma}

\begin{proof}
Fix $x \in D$ and $\eps > 0$. By the definition of $\form_M$, there is a nonzero function $f \in C_c(D) \cap \dom(\form_M)$ such that
\formula{
 \form_M(f) & \le (\lambda_M + \eps) \lV f \rV_2^2 .
}
Let $\tilde V$ be an arbitrary continuous function on $\R^d$ such that $f(x) \tilde V(x) = 0$ for all $x \in \R^d$, and $\tilde V(x) \to \infty$ as $\lv x \rv \to \infty$. Define the mutliplicative functional
\formula{
 \tilde M_t & = \exp \biggl( -\int_0^t \tilde V(X_s) ds \biggr) M_t .
}
The transition semigroup $P_t^{\tilde M}$ is compact on $L^2(D)$; see~\cite{Lenz-Stollmann-Wingert}. Denote by $\ph$ the ground state of $(X_t, \tilde M_t)$, that is, a nonnegative normalised eigenfunction of $P_t^{\tilde M}$ corresponding to the eigenvalue $\exp(-\lambda_{\tilde M} t)$.

Clearly,
\formula{
 \lambda_{\tilde M} & \le \frac{\form_{\tilde M}(f)}{\lV f \rV_2^2} = \frac{\form_M(f)}{\lV f \rV_2^2} \le \lambda_M + \eps ,
}
and so $\lambda_{\tilde M} \le \lambda_M + \eps$. On the other hand, since $M_t \ge \tilde M_t$, we have
\formula{
 P_t^M \ind(x) & \ge P_t^{\tilde M} \ind(x) \ge \frac{P_t^{\tilde M} \ph(x)}{\sup \ph} = \frac{\exp(-\lambda_{\tilde M} t) \ph(x)}{\sup \ph} ,
}
and hence
\formula{
 \liminf_{t \to \infty} \frac{\log P_t^M \ind(x)}{t} & \ge -\lambda_{\tilde M} + \liminf_{t \to \infty} \frac{\log \ph(x) - \log \sup \ph}{t} .
}
By~\eqref{eq:irreducible}, $P_t^M \ph(x) > 0$, and so also $P_t^{\tilde M} \ph(x) > 0$. Thus, $\ph(x) = \exp(\lambda_{\tilde M} t) P_t^{\tilde M} \ph(x) > 0$. By~\eqref{eq:ultracontractive}, $\ph = \exp(\lambda_{\tilde M} t) P_t^{\tilde M} \ph(x) \le \exp(\lambda_{\tilde M} t) P_t^M \ph$ is bounded for $t$ large enough. Hence,
\formula{
 \liminf_{t \to \infty} \frac{\log P_t^M \ind(x)}{t} & \ge -\lambda_{\tilde M} \ge -\lambda_M - \eps .
}
Since $\eps > 0$ was arbitrary, \eqref{eq:lower} follows.
\end{proof}


\subsection{Chernoff--Bennett bound for truncated processes}
\label{sec:bennett}

In this section we fix $r > 0$ and we assume that the Lévy measure $\nu$ is concentrated on $B(0, r) \setminus \{0\}$. In this case $X_t$ has only jumps of magnitude less than $r$ --- it is a finite range Lévy process. Using standard techniques, we prove a crude variant of Chernoff--Bennett inequality (cf.\@ Theorem~2.9 in~\cite{blm}).

The Laplace exponent of $X_t$ is a function $\psi$ satisfying
\formula{
 \ex^x \bigl[ \exp(\xi \cdot X_t) \bigr] & = \exp(\xi \cdot x + t \psi(\xi)) .
} 
It is given by the Lévy--Khintchine formula
\formula{
 \psi(\xi) & = \frac{1}{2} \sum_{i, j = 1}^d a_{ij} \xi_i \xi_j + \int_{B(0, r) \setminus \{0\}} (e^{\xi \cdot x} - 1 - \xi \cdot x) \nu(dx) \\
 & = \frac{1}{2} \sum_{i, j = 1}^d a_{ij} \xi_i \xi_j + \int_{B(0, r) \setminus \{0\}} (\cosh(\xi \cdot x) - 1) \nu(dx) .
}
Under our assumption, $\psi$ is defined on all of $\R^d$.

It is easy to check that for $x \in B(0, r) \setminus \{0\}$, we have
\formula{
 r^2 (\cosh(\xi \cdot x) - 1) \le \lv x \rv^2 (\cosh(r \lv \xi \rv) - 1) .
}
Additionally, $\tfrac{1}{2} r^2 \lv \xi \rv^2 \le \cosh(r \lv \xi \rv) - 1$. Therefore,
\formula[eq:laplace:bound]{
 \psi(\xi) & \le \biggl(\frac{\lV a \rV}{r^2} + \frac{1}{r^2} \int_{B(0, r) \setminus \{0\}} \lv x \rv^2 \nu(dx) \biggr) (\cosh(r \lv \xi \rv) - 1) .
}
Here $\lV a \rV$ denotes the largest eigenvalue of the matrix $a$.

\begin{lemma}
\label{lem:bennett}
Let $r > 0$ and suppose that $\nu$ is concentrated on $B(0, r) \setminus \{0\}$. There are constants $c_1, c_2, c_3 > 0$ such that for $t, R > 0$, we have
\formula{
 \pr^x \bigl[ \lv X_t - x \rv \ge R \bigr] & \le \exp \bigl( c_1 + c_2 R - c_3 R \log (R / t) \bigr) .
}
\end{lemma}

\begin{proof}
For every $\xi \in \R^d$, we have
\formula{
 \ex^x \bigl[ \exp(\xi \cdot (X_t - x)) \bigr] & = \exp(t \psi(\xi)) .
}
By Markov's inequality, for $u > 0$,
\formula{
 \pr^x \bigl[ \xi \cdot (X_t - x) \ge u \lv \xi \rv \bigr] & \le e^{-u \lv \xi \rv} \ex^x \bigl[ \exp(\xi \cdot (X_t - x)) \bigr] = \exp(t \psi(\xi) - u \lv \xi \rv) .
}
Combining this with~\eqref{eq:laplace:bound} and the inequality $\cosh(r \lv \xi \rv) - 1 \le \exp(r \lv \xi \rv)$, we find that
\formula[eq:bennett:markov]{
 \pr^x \bigl[ \xi \cdot (X_t - x) \ge u \lv \xi \rv \bigr] & \le \exp \bigl( \sigma^2 t e^{r \lv \xi \rv} - u \lv \xi \rv \bigr) ,
}
where
\formula{
 \sigma^2 & = \frac{\lV a \rV}{r^2} + \frac{1}{r^2} \int_{B(0, r) \setminus \{0\}} \lv x \rv^2 \nu(dx) .
}
The left-hand side of~\eqref{eq:bennett:markov} does not depend on the length of $\xi$, and for the optimal value $\lv \xi \rv = r^{-1} \log((\sigma^2 r t)^{-1} u)$, the right-hand side is equal to $\exp(r^{-1} u - r^{-1} u \log ((\sigma^2 r t)^{-1} u))$. It follows that this value is an upper bound for arbitrary $\xi \in \R^d$ and $u > 0$:
\formula[eq:bennett:projection]{
 \pr^x \bigl[ \xi \cdot (X_t - x) \ge u \lv \xi \rv \bigr] & \le \exp \biggl( \frac{u}{r} - \frac{u}{r} \log \frac{u}{\sigma^2 r t} \biggr) .
}

Note that if $\lv X_t - x \rv \ge R$, then for some coordinate $j$ we have $\lv e_j \cdot (X_t - x) \rv > d^{-1 / 2} R$. We set $\xi = \pm e_j$ and $u = d^{-1 / 2} R$ in~\eqref{eq:bennett:projection}, and sum over $j = 1, 2, \ldots, d$, to find that
\formula*[eq:bennett:constants]{
 \pr^x \bigl[ \lv X_t - x \rv \ge R \bigr] & \le \sum_{j = 1}^d \Bigl( \pr^x \bigl[ e_j \cdot (X_t - x) \ge d^{-1 / 2} R \bigr] + \pr^x \bigl[ (-e_j) \cdot (X_t - x) \ge d^{-1 / 2} R \bigr] \Bigr) \\
 & \le 2 d \exp \biggl( \frac{R}{d^{1 / 2} r} - \frac{R}{d^{1 / 2} r} \log \frac{R}{d^{1 / 2} \sigma^2 r t} \biggr) \\
 & = 2 d \exp \biggl( \frac{1 + \log(d^{1 / 2} \sigma^2 r)}{d^{1 / 2} r} R - \frac{1}{d^{1 / 2} r} R \log \frac{R}{t} \biggr) ,
}
as desired.
\end{proof}


\subsection{Time-varying domains and potentials}
\label{sec:varying}

In this section we consider time-inhomogeneous multiplicative functionals for our Lévy process $X_t$. More precisely, we fix a deterministic sequence of switching times
\formula{
 0 & = \tau_0 < \tau_1 < \tau_2 < \cdots \to \infty ,
}
a sequence of non-empty open sets $D_1, D_2, \ldots$\,, and a sequence of (time-homogeneous) multiplicative functionals $M_1(t), M_2(t), \ldots$\,, such that $D_n$ is the set of permanent points of $M_n(t)$. Our main goal is to study the process $X_t$ with time-dependent killing mechanism: on the time interval $[\tau_{n - 1}, \tau_n)$, $X_t$ is killed according to the multiplicative functional $M_n(t)$ (and so, in particular, upon exiting the set $D_n$).

In order to state our results rigorously, we often need to consider the process $X_t$ started at a given time $s \ge 0$ rather than at time $0$. In this case we write $\pr^{s, x}$ and $\ex^{s, x}$ for the corresponding probability and expectation. We understand that under $\pr^{s, x}$, $X_t$ is only defined for $t \ge s$.

By $\thet_s$ we denote the usual Markov shift operator, so that $X_t = X_{t - s} \circ \thet_s$ and $M_n(t) = M_n(s) \cdot M_n(t - s) \circ \thet_s$ when $0 \le s \le t$. When $0 \le s \le t$, we denote
\formula{
 M_n(s, t) & = M_n(t - s) \circ \thet_s ,
}
so that $M_n(t) = M_n(s) M_n(s, t)$. Our time-inhomogeneous multiplicative functional $M(t)$ is given by
\formula[eq:functional:varying]{
 M(t) & = \biggl(\prod_{n = 1}^{k - 1} M_n(\tau_{n - 1}, \tau_n) \biggr) M_k(\tau_{k - 1}, t) \qquad \text{if $\tau_{k - 1} \le t < \tau_k$, $k = 1, 2, \ldots$}
}
If $0 \le s \le t$, we define $M(s, t)$ so that $M(t) = M(s) M(s, t)$. More precisely, if $\tau_{k - 1} \le s \le t \le \tau_k$, we set
\formula{
 M(s, t) & = M_k(s, t) ,
}
while if $\tau_{j - 1} \le s < \tau_j \le \tau_{k - 1} \le t < \tau_k$, $j<k$, $j,k = 1, 2, \ldots$\,, then
\formula{
 M(s, t) & = M_j(s, \tau_j) \biggl(\prod_{n = j + 1}^{k - 1} M_n(\tau_{n - 1}, \tau_n) \biggr) M_k(\tau_{k - 1}, t) .
}
We stress that while $M_n(s, t) = M_n(t - s) \circ \thet_s$ for every $n = 1, 2, \ldots$\,, we do not generally have $M(s, t) = M(t - s) \circ \thet_s$.

Even though $M(t)$ is not a (time-homogeneous) multiplicative functional, we denote
\formula[eq:varying:heat]{
 P_t^M f(x) & = \ex^x \bigl[ f(X_t) M(t) \bigr]
}
whenever the integral makes sense. More generally, if $0 \le s \le t$, we write
\formula{
 P_{s, t}^M f(X_s) & = \ex^{s, x} \bigl[ f(X_t) M(s, t) \bigr] .
}
It is straightforward to verify that the pair $(X_t, M(t))$ has the (time-inhomogeneous) Markov property, and hence
\formula{
 \ex^x \bigl[ f(X_t) M(s, t) \big\vert \mathscr F_s \bigr] & = \ex^x \bigl[ f(X_t) M(s, t) \big\vert X_s \bigr] = \ex^{s, X_s} \bigl[ f(X_t) M(s, t) \bigr] = P_{s, t}^M f(X_s) .
}
Therefore,
\formula{
 P_t^M f(x) & = \ex^x \bigl[ f(X_t) M(s) M(s, t) \bigr] \\
 & = \ex^x \bigl[ M(s) \ex^x \bigl[ f(X_t) M(s, t) \big\vert \mathscr F_s \bigr] \bigr] \\
 & = \ex^x \bigl[ M(s) P_{s, t}^M f(X_s) \bigr] = P_s^M P_{s, t}^M f(x) .
}
In the same vein, if $0 \le r \le s \le t$, then
\formula{
 P_{r, t}^M f(x) & = P_{r, s}^M P_{s, t}^M f(x) .
}
Observe that if $\tau_n \le s \le t < \tau_{n + 1}$, then, by time-homogeneity of $X_t$ and $M_n(t)$,
\formula{
 P_{s, t}^M f(x) & = \ex^{s, x} \bigl[ f(X_t) M_n(s, t) \bigr] = \ex^x \bigl[ f(X_{t - s}) M_n(t - s) \bigr] = P_{t - s}^{M_n} f(x) .
}

We define
\formula[eq:varying:lambda]{
 \lambda_M & = \inf\{\lambda_{M_1} , \lambda_{M_2} , \ldots\} ,
}
and, abusing the notation, we say that $\lambda_M$ is the bottom of the spectrum of $(X_t, M(t))$.

\begin{lemma}
\label{lem:varying:l2}
For every $f \in L^2(\R^d)$ and $0 \le s \le t$,
\formula[eq:varying:l2]{
 \lV P_{s, t}^M f \rV_2 & \le \exp(-\lambda_M (t - s)) \lV f \rV_2 .
}
\end{lemma}

\begin{proof}
We proceed by backward induction. Fix $t \ge 0$, and choose $k = 1, 2, \ldots$ such that $\tau_{k - 1} \le t < \tau_k$. If $\tau_{k - 1} \le s \le t$, then
\formula{
 \lV P_{s, t}^M f \rV_2 & = \lV P_{t - s}^{M_k} f \rV_2 \le \exp(-\lambda_{M_k} (t - s)) \lV f \rV_2 \le \exp(-\lambda_M (t - s)) \lV f \rV_2 ,
}
as desired.

Suppose that for some $j = 1, 2, \ldots, k - 1$, the estimate~\eqref{eq:varying:l2} holds when $\tau_j \le s \le t$, and fix $s$ such that that $\tau_{j - 1} \le s < \tau_j$. We have
\formula{
 \lV P_{s, t}^M f \rV_2 & = \lV P_{s, \tau_j}^M P_{\tau_j, t}^M f \rV_2 \\
 & = \lV P_{\tau_j - s}^{M_j} P_{\tau_j, t}^M f \rV_2 \\
 & \le \exp(-\lambda_{M_j} (\tau_j - s)) \lV P_{\tau_j, t}^M f \rV_2 \\
 & \le \exp(-\lambda_M (\tau_j - s)) \lV P_{\tau_j, t}^M f \rV_2 .
}
By the inductive hypothesis,
\formula{
 \lV P_{s, t}^M f \rV_2 & \le \exp(-\lambda_M (\tau_j - s)) \exp(-\lambda_M (t - \tau_j)) \lV f \rV_2 ,
}
that is, \eqref{eq:varying:l2} holds true when $\tau_j \le s \le t$. This completes the proof by induction.
\end{proof}

\begin{lemma}
\label{lem:varying}
If~\eqref{eq:ultracontractive} holds true, then for every nonnegative $f \in L^2(\R^d)$, $x \in \R^d$ and $0 \le s \le t$,
\formula{
 P_t^M f(x) & \le \lV P_s \rV_{2, \infty} \exp(-\lambda_M (t - s)) \lV f \rV_2 .
}
\end{lemma}

\begin{proof}
We have
\formula{
 P_t^M f(x) & = P_s^M P_{s, t}^M f(x) \le P_s P_{s, t}^M f(x) \le \lV P_s \rV_{2, \infty} \lV P_{s, t}^M f \rV_2 .
}
The result thus follows from Lemma~\ref{lem:varying:l2}.
\end{proof}


\subsection{Decomposition of a Lévy process}
\label{sec:decomposition}

For a given $r > 0$, we decompose the Lévy process $X_t$ into two independent parts: a finite range Lévy process $\tilde X_t$, which includes the Brownian part of $X_t$ together with all jumps of $X_t$ of magnitude less than $r$, and a compound Poisson process $Z_t$, whose jumps coincide with the jumps of $X_t$ of magnitude at least $r$. We assume that $\tilde X_0 = X_0$ and $Z_0 = 0$.

By putting a tilde over a symbol we indicate that it pertains to $\tilde X_t$ instead of $X_t$. For example, we denote by $\tilde \tau_D$ the first exit time of $\tilde X_t$ from $D$. If $M_t$ is a Feynman--Kac functional~\eqref{eq:feynman:kac}, then we define a similar functional for $\tilde X_t$ by the formula
\formula{
 \tilde M_t & = \begin{cases} \displaystyle \exp\biggl(-\int_0^t V(\tilde X_s) ds\biggr) & \text{if $t < \tilde \tau_D$;} \\ 0 & \text{if $t \ge \tilde \tau_D$,} \end{cases}
}
and we also consider a version shifted by $y \in \R^d$:
\formula{
 \tilde M_t^y & = \begin{cases} \displaystyle \exp\biggl(-\int_0^t V(\tilde X_s + y) ds\biggr) & \text{if $t < \tilde \tau_{D - y}$;} \\ 0 & \text{if $t \ge \tilde \tau_{D - y}$.} \end{cases}
}
In the more general case when $V$ is a measure, we define $\tilde M_t$ and $\tilde M_t^y$ in the same vein, so that the corresponding Revuz measures are $V$ and $V(\cdot + y)$. Observe that the bottom of the spectrum of $(\tilde X_t, \tilde M^y_t)$ does not depend on $y$. We denote it by $\tilde \lambda_{\tilde M}$; the two tildes indicate that it is defined using the process $\tilde X_t$ with the (time-homogeneous) multiplicative functional $\tilde M$.

Since the processes $\tilde X_t$ and $Z_t$ are independent, we may assume that the probability $\pr^x$ is the product measure $\tilde \pr^x \times \pr_Z$, where $\tilde \pr^x$ is the probability for $\tilde X_t$, and $\pr_Z$ corresponds to $Z_t$ (since $Z_t = 0$, there is no need for the intial point in the notation $\pr_Z$). We can imagine $\tilde \ex^x$ to be the conditional expectation
\formula{
 \tilde \ex^x \bigl[ \Phi \bigr] & = \ex^x \big[ \Phi \big\vert Z \bigr] ,
}
which averages over all paths of $\tilde X_t$, but keeps the path of $Z_t$ random, but it will be more convenient to use the product space structure. Formally, we assume that $\Omega = \tilde \Omega \times \Omega_Z$, and we denote the elementary events as $\omega = (\tilde \omega, \omega_Z)$.

We fix a path of $Z_t$ (that is, fix $\omega_Z$), and we denote by $0 < \tau_1 < \tau_2 < \cdots \to \infty$ all times at which $Z_t$ makes a jump. We also let $\tau_0 = 0$. By $Y_n = Z_{\tau_{n - 1}}$ we denote the position of $Z_t$ before the $n$th jump. Observe that $X_t = \tilde{X}_t + Y_n$ when $\tau_{n - 1} \le t < \tau_n$, and so
\formula{
 M_t & = \biggl(\prod_{n = 1}^{k - 1} \tilde M^{Y_n}_{\tau_n - \tau_{n - 1}} \circ \thet_{\tau_{n - 1}} \biggr) \tilde M^{Y_k}_{t - \tau_{k - 1}} \circ \thet_{\tau_{k - 1}} \qquad \text{if $\tau_{k - 1} \le t < \tau_k$, $k = 1, 2, \ldots$}
}
That is, conditionally on the path of $Z_t$, the process $M_t$ can be viewed as a time-inhomogeneous multiplicative functional of $\tilde X_t$, as in~\eqref{eq:functional:varying}.

The bottom of the spectrum of $(\tilde X_t, \tilde M_t^y)$ does not depend on $y$ and it is equal to $\tilde \lambda_{\tilde M}$. Therefore, the bottom of the spectrum $\tilde \lambda_M$ of $(\tilde X_t, M_t)$, defined in~\eqref{eq:varying:lambda}, is equal to $\tilde \lambda_{\tilde M}$, the bottom of the spectrum of $(\tilde X_t, \tilde M_t)$.

Applying Lemma~\ref{lem:varying} to the pair $(\tilde X_t, M_t)$, with a fixed path of $Z_t$, we obtain the following result: if $0 \le s \le t$ and $f \in L^2(\R^d)$ is nonnegative, then
\formula{
 \tilde P_t^M f(x) & \le \lV \tilde P_s \rV_{2, \infty} \exp(-\tilde \lambda_M (t - s)) \lV f \rV_2 .
}
Here $\tilde P_t$ are the transition operators of $\tilde X_t$, $\tilde \lambda_M = \tilde \lambda_{\tilde M}$, and $\tilde P_t^M$ is defined as in~\eqref{eq:varying:heat}:
\formula{
 \tilde P_t^M f(x) & = \tilde \ex^x \bigl[ f(\tilde X_t) M_t \bigr] .
}
Note that the definition of $M$, and hence also $P_t^M$, depends on the path of $Z_t$ (that is, they are functions of $\omega_Z$). Averaging over all possible paths, we find that
\formula{
 \ex^x \bigl[ f(\tilde X_t) M_t \bigr] & = \ex_Z \bigl[ \tilde \ex^x \bigl[ f(\tilde X_t) M_t \bigr] \bigr] = \ex_Z \bigl[ \tilde P_t^M f(x) \bigr] \le \lV \tilde P_s \rV_{2, \infty} \exp(-\tilde \lambda_{\tilde M} (t - s)) \lV f \rV_2 .
}
This proves the following result.

\begin{corollary}
\label{cor:truncated}
Fix $r > 0$, and let $\tilde X_t$ be the process $X_t$ with all jumps of magnitude at least $r$ removed. Denote by $\tilde P_t$ the transition operators of $\tilde X_t$, by $\tilde M_t$ the mutliplicative functional of $\tilde X_t$ corresponding to the potential $V(x)$ (or the Revuz measure $V(dx)$) of the functional $M_t$, and by $\tilde \lambda_M$ the bottom of the spectrum of $(\tilde X_t, \tilde M_t)$. If $\tilde X_t$ satisfies~\eqref{eq:ultracontractive}, then for every nonnegative $f \in L^2(\R^d)$, $x \in \R^d$ and $0 \le s \le t$,
\formula{
 \ex^x \bigl[ f(\tilde X_t) M_t \bigr] & \le \lV \tilde P_s \rV_{2, \infty} \exp(-\tilde \lambda_M (t - s)) \lV f \rV_2 .
}
\end{corollary}

The Dirichlet form of the finite range Lévy process $\tilde X_t$ is given by
\formula{
 \tilde \form(f) & = \frac{1}{2} \int_{\R^d} \sum_{i, j = 1}^d a_{ij} \frac{\partial^2 f}{\partial x_i \partial x_j}(x) dx + \frac{1}{2} \int_{\R^d} \int_{B(0, r) \setminus \{0\}} (f(x + y) - f(x))^2 \nu(dy) dx ,
}
and that of the compound Poisson process $Z_t$ is
\formula{
 \form(f) - \tilde \form(f) & = \frac{1}{2} \int_{\R^d} \int_{\R^d \setminus B(0, r)} (f(x + y) - f(x))^2 \nu(dy) dx .
}
Using $(f(x + y) - f(x))^2 \le 2 (f(x + y))^2 + 2 (f(x))^2$, we find that
\formula{
 \form(f) - \tilde \form(f) & \le \int_{\R^d} \int_{\R^d \setminus B(0, r)} \bigl( (f(x + y))^2 + (f(x))^2 \bigr) \nu(dy) dx = 2 \nu(\R^d \setminus B(0, r)) \lV f \rV_2^2 .
}
Since the multiplicative functional $M_t$ of $X_t$ and the multiplicative functional $\tilde M_t$ of $\tilde X_t$ correspond to the same potential (or Revuz measure) $V$, we have
\formula{
 \tilde \form_{\tilde M}(f) - \tilde \form(f) & = \form_M(f) - \form(f) = \int_D \lv f(x) \rv^2 V(dx) .
}
It follows that
\formula{
 \form_M(f) - \tilde \form_{\tilde M}(f) & = \form(f) - \tilde \form(f) .
}
Recall that the bottom of the spectrum of $(\tilde X_t, \tilde M_t)$ on $D$ is equal to $\tilde \lambda_M$. By definition,
\formula{
 \tilde \lambda_M = \tilde \lambda_{\tilde M} = \inf_f \frac{\tilde \form_{\tilde M}(f)}{\lV f \rV_2^2} & = \inf_f \frac{\form_M(f) - (\form_M(f) - \tilde \form_{\tilde M}(f))}{\lV f \rV_2^2} = \inf_f \frac{\form_M(f) - (\form(f) - \tilde \form(f))}{\lV f \rV_2^2} \\
 & \ge \inf_f \frac{\form_M(f) - 2 \nu(\R^d \setminus B(0, r)) \lV f \rV_2^2}{\lV f \rV_2^2} = \lambda_M - 2 \nu(\R^d \setminus B(0, r)) ,
}
where the infimum is taken over all $f \in C_c^\infty(D)$. It follows that
\formula[eq:truncated:lambda]{
 \tilde \lambda_M & \ge \lambda_M - 2 \nu(\R^d \setminus B(0, r)) .
}


\subsection{Upper bound}
\label{sec:upper}

We are ready to prove the upper bound in~\eqref{eq:goal}.

\begin{lemma}
\label{lem:upper}
If~\eqref{eq:ultracontractive} holds true, then for every $x \in D$,
\formula{
 \limsup_{t \to \infty} \frac{\log P_t^M \ind(x)}{t} & \le -\lambda_M .
}
\end{lemma}

\begin{proof}
Observe that $X_t = \tilde X_t$ with positive probability $q_t$, and so for $f \ge 0$ we have $P_t f(x) \ge q_t \tilde P_t f(x)$. In particular, if $P_t f$ is bounded, then $\tilde P_t f$ is bounded. This property extends to arbitrary $f \in L^2(\R^d)$, and it follows that if~\eqref{eq:ultracontractive} holds true for $X_t$, then it also holds for $\tilde X_t$.

We continue to use the notation introduced in the previous section. We fix $s \ge 0$ such that $\lV \tilde P_s \rV_{2, \infty}$ is finite, and we consider $R = t^2$. For $x \in \R^d$ and $t > s$, we write
\formula[eq:upper:decomposition]{
 P_t^M \ind(x) & = \ex^x \bigl[ M_t \bigr] = \ex^x \bigl[ \ind_{B(x, R)}(\tilde X_t) M_t \bigr] + \ex^x \bigl[ \ind_{\R^d \setminus B(x, R)}(\tilde X_t) M_t \bigr] .
}
We estimate the former term using Corollary~\ref{cor:truncated}, applied to $f = \ind_{B(x, R)}$. Since
\formula{
 \lV \ind_{B(x, R)} \rV_2 & = \lv B(0, 1) \rv^{1 / 2} R^{d / 2} = \lv B(0, 1) \rv^{1 / 2} t^d ,
}
we have
\formula{
 \ex^x \bigl[ \ind_{B(x, R)}(\tilde X_t) M_t \bigr] & \le \lV \tilde P_s \rV_{2, \infty} \exp(-\tilde \lambda_{\tilde M} (t - s)) \lv B(0, 1) \rv^{1 / 2} t^d .
}
Clearly,
\formula{
 \limsup_{t \to \infty} \frac{1}{t} \log \bigl( \ex^x \bigl[ \ind_{B(x, R)}(\tilde X_t) M_t \bigr] \bigr) & \le -\tilde \lambda_{\tilde M} .
}
The latter term in~\eqref{eq:upper:decomposition} is controlled by Lemma~\ref{lem:bennett}. We have
\formula{
 \ex^x \bigl[ \ind_{\R^d \setminus B(x, R)}(\tilde X_t) M_t \bigr] & \le \pr^x \bigl[ \lv \tilde X_t - x \rv \ge R \bigr] \\
 & \le \exp \bigl( c_1 + c_2 R - c_3 R \log (R / t) \bigr) \\
 & = \exp \bigl( c_1 + c_2 t^2 - c_3 t^2 \log t \bigr) ,
}
and hence
\formula{
 \limsup_{t \to \infty} \frac{1}{t} \log \bigl( \ex^x \bigl[ \ind_{\R^d \setminus B(x, R)}(\tilde X_t) M_t \bigr] \bigr) & \le \lim_{t \to \infty} \frac{c_1 + c_2 t^2 - c_3 t^2 \log t}{t} = -\infty .
}
Combining the above estimates with~\eqref{eq:upper:decomposition}, and using~\eqref{eq:truncated:lambda}, we conclude that
\formula{
 \limsup_{t \to \infty} \frac{P_t^M \ind(x)}{t} & \le -\tilde \lambda_{\tilde M} \le -\lambda_M + 2 \nu(\R^d \setminus B(0, r)) .
}
The desired result follows, because $\nu(\R^d \setminus B(0, r))$ can be made arbitrarily small by taking $r$ large enough.
\end{proof}

\begin{proof}[Proof of Theorem~\ref{thm:general}]
This is an immediate consequence of Lemmas~\ref{lem:lower} and~\ref{lem:upper}.
\end{proof}

%
%

\section{Isotropic stable Lévy processes}\label{sec:stable}

Let $\alpha\in\left(0,2\right)$. In this section we prove Theorem~\ref{thm:stable}. We consider the
case where $X_{t}$ is the isotropic $\alpha$-stable Lévy process
in $\mathbb{R}^{d}$ with 
\[
\nu\left(x\right)=c_{d,\alpha}\left|x\right|^{-d-\alpha},
\]
where $c_{d,\alpha}=2^{\alpha}\pi^{-d/2}\Gamma(\tfrac{d+\alpha}{2})|\Gamma(-\tfrac{\alpha}{2})|^{-1}$
and $a_{i,j}=0$ for $i,j=1,2,\dots,d$. 
\begin{lem}
\label{lem:HK-scaling}Let $\tilde{p}^{r}\left(t,x,y\right)$ be the
heat kernel for the finite range Lévy process $\tilde{X}_{t}$
with jumps at most $r$. Then, the following scaling property holds:
\[
\tilde{p}^{r}\left(t,x,y\right)=\frac{1}{r^d}\tilde{p}^{1}\left(\frac{t}{r^\alpha},\frac{x}{r},\frac{y}{r}\right)
\]
for all $t>0$ and $x,y\in\mathbb{R}^{d}$. 
\end{lem}

\begin{proof}
We include the details here for completeness. One can alternatively look at  \cite[Proposition 11.10]{sato1999levy}. 
Let $\phi_{r}$ be the characteristic (or Fourier) exponent of $\tilde{X}_{t}$; that is,
\[
\mathbb{E}^{x}\left[\exp\left(i\xi\cdot\left(\tilde{X}_{t}-x\right)\right)\right]=\exp\left(-t\phi_{r}\left(\xi\right)\right).
\]
Using the Lévy--Khintchine formula and the symmetry of $\nu$, one has that 
\begin{align*}
\phi_{r}\left(\xi\right) & =c_{d,\alpha}\int_{\left|x\right|<r}\left(1-\cos\left(\xi\cdot x\right)\right)\left|x\right|^{-d-\alpha}dx\\
 & =c_{d,\alpha}\int_{\frac{\left|ry\right|}{r}<1}\left(1-\cos\left(\xi\cdot ry\right)\right)\left|ry\right|^{-d-\alpha}r^{d}dy\\
 & =r^{-\alpha}c_{d,\alpha}\int_{\left|y\right|<1}\left(1-\cos\left(r\xi\cdot y\right)\right)\left|y\right|^{-d-\alpha}dy\\
 & =r^{-\alpha}\phi_{1}\left(r\xi\right).
\end{align*}
By using inverse Fourier transform with the substitution $\eta=r\xi$,
we have 
\begin{align*}
\tilde{p}^{r}\left(t,x,y\right) & =\left(2\pi\right)^{-d}\int_{\mathbb{R}^{d}}e^{-i\left(y-x\right)\cdot\xi}e^{-t\phi_{r}\left(\xi\right)}d\xi\\
 & =\left(2\pi\right)^{-d}\int_{\mathbb{R}^{d}}e^{-i\left(\frac{y-x}{r}\right)\cdot r\xi}e^{-\left(tr^{-\alpha}\right)\phi_{1}\left(r\xi\right)}d\xi\\
 & =\left(2\pi\right)^{-d}\int_{\mathbb{R}^{d}}e^{-i\left(\frac{y-x}{r}\right)\cdot\eta}e^{-\left(tr^{-\alpha}\right)\phi_{1}\left(\eta\right)}r^{-d}d\eta\\
 & =r^{-d}\tilde{p}^{1}\left(tr^{-\alpha},\frac{x}{r},\frac{y}{r}\right),
\end{align*}
as desired.
\end{proof}
We can now prove \eqref{MainEstimate-1} in the following theorem. 
\begin{thm}
Consider the isotropic $\alpha$-stable Lévy process with index $\alpha\in\left(0,2\right)$ and let $V$  be a non-negative, locally bounded Borel function on $D$.
For $t>1$ there exists a constant $C\geq1$ independent of $t$ (depending only on $d$, $\alpha$ and $\lambda_M$), such that 
\begin{equation}
P_{t}^{M}\ind\left(x\right)= \ex^x \biggl[ \exp \biggl( -\int_0^t V(X_s) ds \biggr) \ind_{\{\tau_D > t\}} \biggr] \leq   C\left(1+t\right)^{\frac{d}{2\alpha}+\frac{d}{2}}e^{-\lambda_D t}.
\end{equation}

\end{thm}

\begin{proof}
Let $x\in D$ and $t>0$. For
\[
 r = t^{1 / \alpha} ,
\]
we consider the finite-range process $\tilde X_t$ from Lemma~\ref{lem:HK-scaling}. As in~\eqref{eq:upper:decomposition}, we consider the following decomposition:
\begin{align}
P_{t}^{M}\ind\left(x\right) & =\mathbb{E}^{x}\left[1_{B(x,R)}\left(\tilde{X}_t\right)M_{t}\right]+\mathbb{E}^{x}\left[1_{\mathbb{R}^{d}\setminus B(x,R)}\left(\tilde{X}_t\right)M_{t}\right]\nonumber \\
 & =I+II.\label{eq:Thm2pf-0}
\end{align}
Recall that $\tilde \lambda_M$ denotes the bottom of spectrum of $(\tilde X_t, M_t)$, where $M_t$ is treated as a time-inhomogeneous additive functional of $\tilde X_t$ (conditionally on the path of the compound Poisson process $X_t - \tilde X_t$).

\textbf{Step 1:} Bounding $\tilde{\lambda}_{M}$.

Recall that 
\begin{align*}
\tilde{\lambda}_M & \geq\lambda_{M}-2\nu\left(\mathbb{R}^{d}\setminus B\left(0,r\right)\right)\\
 & =\lambda_{M}-2c_{d,\alpha}^{\prime}r^{-\alpha};
\end{align*}
see~\eqref{eq:truncated:lambda}. Since $r = t^{1 / \alpha}$, we get
\begin{equation}
\tilde{\lambda}_M\geq\lambda_{M}-c_{0}t^{-1},\label{eq:Thm2pf-1}
\end{equation}
so that for every $t>s>0$ we have 
\begin{align}
e^{-\tilde{\lambda}_M\left(t-s\right)} & \leq e^{-\left(\lambda_{M}-c_{0}t^{-1}\right)\left(t-s\right)}\nonumber \\
 & =e^{-t\lambda_{M}+c_{0}+s\lambda_{M}-sc_{0}t^{-1}}\nonumber \\
 & \leq e^{c_{0}+\lambda_{M}s}e^{-\lambda_{M}t}.\label{eq:Thm2pf-2}
\end{align}

\textbf{Step 2:} Bounding $I$.

Choose $R=art$ for a constant $a>1$ (depending only on $d, \alpha$ and $\lambda_M$) to be chosen later. By Corollary \ref{cor:truncated} applied
to $f=1_{B(x,R)}$ together with \eqref{eq:Thm2pf-2},
we have 
\begin{align}
\mathbb{E}^{x}\left[1_{B(x,R)}\left(\tilde{X}_t\right)M_{t}\right] & \leq\left\Vert \tilde{P}_{s}\right\Vert _{2,\infty}\left\Vert 1_{B(x,R)}\right\Vert _{2}e^{-\tilde{\lambda}_M\left(t-s\right)}\nonumber \\
 & =\left\Vert \tilde{P}_{s}\right\Vert _{2,\infty}\left|B\left(0,1\right)\right|R^{d/2}e^{-\tilde{\lambda}_M\left(t-s\right)}\nonumber \\
 & \leq\left\Vert \tilde{P}_{s}\right\Vert _{2,\infty}\left|B\left(0,1\right)\right|R^{d/2}e^{c_{0}+\lambda_{M}s}e^{-\lambda_{M}t}.\label{eq:Thm2pf-3}
\end{align}
We will need a bound on 
\[
\left\Vert \tilde{P}_{s}\right\Vert _{2,\infty}=\left\Vert \tilde{p}_{s}^{r}\right\Vert _{2},
\]
where $\tilde{p}_{s}^{r}$ is the heat kernel of $\tilde{P}_{s}$
(and we now note that dependence on $r$). By the scaling result from
Lemma \ref{lem:HK-scaling} and Chapman Kolmogorov, we have that
\begin{align}
\left\Vert \tilde{P}_{s}\right\Vert _{2,\infty}^{2} & =\sup_{x\in\mathbb{R}^{d}}\left(\int_{\mathbb{R}^{d}}\tilde{p}^{r}\left(s,x,y\right)^{2}dy\right)\nonumber \\
 & =\sup_{x\in\mathbb{R}^{d}}\tilde{p}^{r}\left(2s,x,x\right)\nonumber \\
 & =\sup_{x\in\mathbb{R}^{d}}r^{-d}\tilde{p}^{1}\left(2sr^{-\alpha},\frac{x}{r},\frac{x}{r}\right)\label{eq:Thm2pf-3b}
\end{align}
In \cite[Proposition~2.1]{Chen-Kumagai-Kim} it was shown that for all
$u\in(0,2]$ there exists a constant $C>0$ such that we
have 
\begin{equation}
\tilde{p}^{(1)}\left(u,x,y\right)\leq Cu^{-d/\alpha}.\label{eq:Thm2pf-4}
\end{equation}
Applying \eqref{eq:Thm2pf-4} to \eqref{eq:Thm2pf-3b}
with $u=2sr^{-\alpha}$ (note that $u = 2 s / t \le 2$), we have that
\begin{equation}
\left\Vert \tilde{P}_{s}^{r}\right\Vert _{2,\infty}^{2}\leq r^{-d}C\left(2sr^{-\alpha}\right)^{-d/\alpha}=C^{\prime}s^{-d/\alpha}.\label{eq:Thm2pf-5}
\end{equation}
Recall that $R = a r t =a t^{1 + 1 / \alpha}$. Using \eqref{eq:Thm2pf-5} in \eqref{eq:Thm2pf-3},
we have
\begin{align*}
I=\mathbb{E}^{x}\left[1_{B(x,R)}\left(\tilde{X}_t\right)M_{t}\right] & \leq\left\Vert \tilde{P}_{s}\right\Vert _{2,\infty}\left|B\left(0,1\right)\right|R^{d/2}e^{c_{0}+\lambda_{M}s}e^{-\lambda_{M}t}\\
 & \leq Cs^{-\frac{d}{2\alpha}}\left|B\left(0,1\right)\right|a^{d/2}t^{\frac{d}{2\alpha}+\frac{d}{2}}e^{c_{0}+\lambda_{M}s}e^{-\lambda_{M}t}.
\end{align*}
Choosing $s=1$ we have that there exists a constant $C_{d,\alpha}>0$
such that for all $t>1$,
\begin{equation}
I\leq C_{d,\alpha}a^{d/2}e^{\lambda_{M}}t^{\frac{d}{2\alpha}+\frac{d}{2}}e^{-\lambda_{M}t}.\label{eq:Thm2pf-5b}
\end{equation}

\textbf{Step 3:} Bounding $II$

Note that by Lemma 3.2, we have that there exist constants $c_{1},c_{2},c_{3}>0$
such that for $t,R>0$ we have
\begin{align*}
II=\mathbb{E}^{x}\left[1_{\mathbb{R}^{d}\setminus B(x,R)}\left(\tilde{X}_t\right)M_{t}\right] & \leq\mathbb{P}^{x}\left[\left|\tilde{X}_t-x\right|\geq R\right]\\
 & \leq\exp\left(c_{1}+c_{2}R-c_{3}R\log\left(\frac{R}{t}\right)\right).
\end{align*}
From the proof Lemma 3.2 (see~\eqref{eq:bennett:constants}), the constants $c_{1},c_{2},c_{3}$ are given by
\[
c_1 = \log(2 d), \qquad c_{2}=\frac{1+\log\left(r\sigma^{2}\sqrt{d}\right)}{r\sqrt{d}},\qquad\text{and }c_{3}=\frac{1}{r\sqrt{d}},
\]
where 
\begin{align*}
\sigma^{2} & =\frac{\left\Vert a\right\Vert }{r^{2}}+\frac{1}{r^{2}}\int_{B\left(0,r\right)\setminus\left\{ 0\right\} }\left|x\right|^{2}d\nu\left(dx\right)\\
 & =0+\frac{1}{r^{2}}c_{d,\alpha}\int_{\left|x\right|<r}\left|x\right|^{2-d-\alpha}dx\\
 & =\frac{1}{r^{2}}c_{d,\alpha}'\int_{0}^{r}y^{2-d-\alpha}y^{d-1}dy\\
 & =\frac{1}{r^{2}}c_{d,\alpha}''r^{2-\alpha}\\
 & =c_{d,\alpha}''r^{-\alpha}.
\end{align*}

Plugging this in with $R=art$ and $r=t^{1/\alpha}$ we have that
\begin{align*}
II & \leq 2d \exp\left(\frac{1+\log\left(r\sigma^{2}\sqrt{d}\right)}{r\sqrt{d}}R-\frac{1}{r\sqrt{d}}R\log\left(\frac{R}{t}\right)\right)\\
 & = 2d \exp\left(\frac{1+\log\left(rc_{d,\alpha}''r^{-\alpha}\sqrt{d}\right)}{r\sqrt{d}}art-\frac{1}{r\sqrt{d}}art\log\left(\frac{art}{t}\right)\right)\\
 & = 2d \exp\left(\frac{c_{d,\alpha}'''+\log\left(t^{1/\alpha-1}\right)-\log\left(at^{1/\alpha}\right)}{\sqrt{d}}at\right)\\
 & = 2d \exp\left(\left(c_{d,\alpha}'''-\log\left(t\right)-\log\left(a\right)\right)\frac{at}{\sqrt{d}}\right)\\
 & \leq 2d \exp\left(\left(c_{d,\alpha}'''-\log\left(a\right)\right)\frac{at}{\sqrt{d}}\right)
\end{align*}
Choosing $a>0$ big enough so that $a c_{d,\alpha}'''+\lambda_{M}\sqrt{d}\leq a\log\left(a\right)$
shows that 
\begin{equation}
II\leq 2de^{-\lambda_{M}t}.\label{eq:Thm2pf-6}
\end{equation}

\textbf{Step 4:} Using \eqref{eq:Thm2pf-5b} and \eqref{eq:Thm2pf-6}
in \eqref{eq:Thm2pf-0} we have shown that there is a constant $C_{d, \alpha} \geq 1$ and an appropriate choice of $a > 0$ such that for all $t\geq1$,
\begin{align*}
P_{t}^{M}1\left(x\right) & \leq C_{d,\alpha}a^{d/2}e^{\lambda_{M}}t^{\frac{d}{2\alpha}+\frac{d}{2}}e^{-\lambda_{M}t}+2de^{-\lambda_{M}t}\\
 & \leq C_{d, \alpha, \lambda_{M}}' t^{\frac{d}{2\alpha}+\frac{d}{2}}e^{-\lambda_{M}t},
\end{align*}
as desired. 
\end{proof}

The proof of \eqref{MainEstimate-1} follows from \eqref{MainEstimate-2}. 

\section{Examples}\label{sec:examples}
In this section, we present several concrete examples of unbounded domains with positive bottom of spectrum to which our results apply. The examples are chosen to illustrate rather different geometric behaviors: a spiral region of finite width, a periodic domain with infinitely many holes, and an Anderson-type model. In particular, positivity of the bottom of spectrum does not require the domain to be bounded, convex, or simply connected. 

For the symmetric $\alpha$ stable process in $\mathbb R^2$, we recall two useful geometric criteria for positivity of the bottom of the spectrum. The first relates the bottom of the spectrum to the inradius of the domain. By a result of Bianchi and Brasco \cite[Theorem 1.1]{Bianchi-Brasco-2022}, if $1<\alpha<2$ and $D\subset\mathbb R^2$ is simply connected with finite inradius
\[ 
R_D=\sup_{x\in D}d(x,D^c), 
\]
then the fractional Makai--Hayman inequality gives 
\[
\lambda_D\geq C_\alpha R_D^{-\alpha}. 
\]
Another useful criterion is given by Chowdhury and Roy \cite[Theorem 1.3(i)]{ChowdhuryRoy}. For the full range $0<\alpha<2$, it is enough to assume that there exist constants $R,c>0$ such that
\begin{equation}\label{eq:uniform-complement} 
|D^c\cap B(x,R)|\geq c,\qquad x\in D. 
\end{equation} 
Condition \eqref{eq:uniform-complement} is easy to verify for a large class of domains, including tubular neighborhoods of curves resembling doodle drawings. Below we give a few examples to illustrate this.

\begin{example}[An unbounded spiral band, see Figure \ref{fig:spiral-band}] 
Let 
\[ 
\Gamma(\theta) =\{(r\cos \theta,r\sin\theta):r=a\theta, \theta\ge \theta_0\}, 
\]
be an Archimedean spiral, where $a>0$. Let $n(\theta)$ denote a choice of unit normal vector to $\Gamma$. Let $h(\theta)>0$ be a continuous function satisfying
\[
\lim_{\theta\to\infty}h(\theta)=h
\]
for some $h>0$. For $h>0$ sufficiently small and $\theta_0$ sufficiently large,   
\[ 
D_{\rm sp} =: \left\{ \Gamma(\theta)+r n(\theta): \theta\ge \theta_0,\ |r|<h(\theta) \right\} 
\] 
is an embedded tubular neighbourhood of $\Gamma$ and has no self intersections. At $\theta=\theta_0$ we close the domain by a smooth cap. Thus $D_{\rm sp}$ is an unbounded, simply connected and nonconvex domain which winds around the origin infinitely many times. Moreover, it is easy to verify that $D_{\rm sp}$ satisfies the condition in \cite[Theorem~1.3(i)]{ChowdhuryRoy}, and therefore \[ \lambda_{D_{\rm sp}}>0 \] for the symmetric $\alpha$ stable process for every $0<\alpha<2$.
\end{example}

\begin{example}[A periodic domain with square holes, see Figure \ref{fig:square-holes}]
Let
\[
Q_{k}
=
k+\left[-\frac14,\frac14\right]^2,
\qquad k\in\mathbb Z^2,
\]
and define
\[
D_{\rm per}
=
\mathbb R^2\setminus\bigcup_{k\in\mathbb Z^2}Q_k.
\]
Then $D_{\rm per}$ is an unbounded, connected and periodic domain with infinitely many holes. Moreover, it is easy to verify that $D_{\rm per}$ satisfies the condition \eqref{eq:uniform-complement}. Indeed, every point of $D_{\rm per}$ lies within a uniformly bounded distance from one of the square holes. Hence there exist constants $R,c>0$ such that
\[
|D_{\rm per}^c\cap B(x,R)|\geq c,
\qquad x\in D_{\rm per}.
\]
Therefore $\lambda_{D_{\rm per}}>0$ for the symmetric $\alpha$-stable process for every $0<\alpha<2$.
\end{example}

\begin{figure}[ht]
    \centering

    \begin{subfigure}[t]{0.48\textwidth}
        \centering
\begin{tikzpicture}[scale=0.25]

    \def\a{0.75}          
    \def\tmin{0.45}       
    \def\tmax{17.9}       
    \def\Xmax{8.2}        

    \pgfmathsetmacro{\xzero}{\a*\tmin*cos(deg(\tmin))}
    \pgfmathsetmacro{\yzero}{\a*\tmin*sin(deg(\tmin))}
    \pgfmathsetmacro{\hzero}{0.48 - 0.10/(\tmin+1)}
    \pgfmathsetmacro{\denzero}{sqrt(1+\tmin*\tmin)}
    \pgfmathsetmacro{\Tx}{(cos(deg(\tmin))-\tmin*sin(deg(\tmin)))/\denzero}
    \pgfmathsetmacro{\Ty}{(sin(deg(\tmin))+\tmin*cos(deg(\tmin)))/\denzero}
    \pgfmathsetmacro{\Nx}{(-sin(deg(\tmin))-\tmin*cos(deg(\tmin)))/\denzero}
    \pgfmathsetmacro{\Ny}{( cos(deg(\tmin))-\tmin*sin(deg(\tmin)))/\denzero}

    \begin{scope}
        \clip (-\Xmax,-\Xmax) rectangle (\Xmax,\Xmax);

        \fill[gray!30]
            plot[
                domain=\tmin:\tmax,
                samples=500,
                variable=\t
            ]
            (
                {
                    \a*\t*cos(deg(\t))
                    +(0.48-0.10/(\t+1))
                    *(-sin(deg(\t))-\t*cos(deg(\t)))
                    /sqrt(1+\t*\t)
                },
                {
                    \a*\t*sin(deg(\t))
                    +(0.48-0.10/(\t+1))
                    *( cos(deg(\t))-\t*sin(deg(\t)))
                    /sqrt(1+\t*\t)
                }
            )
            --
            plot[
                domain=\tmax:\tmin,
                samples=500,
                variable=\t
            ]
            (
                {
                    \a*\t*cos(deg(\t))
                    -(0.48-0.10/(\t+1))
                    *(-sin(deg(\t))-\t*cos(deg(\t)))
                    /sqrt(1+\t*\t)
                },
                {
                    \a*\t*sin(deg(\t))
                    -(0.48-0.10/(\t+1))
                    *( cos(deg(\t))-\t*sin(deg(\t)))
                    /sqrt(1+\t*\t)
                }
            )
            --
            plot[
                domain=180:0,
                samples=100,
                variable=\p
            ]
            (
                {
                    \xzero
                    + \hzero*cos(\p)*\Nx
                    - \hzero*sin(\p)*\Tx
                },
                {
                    \yzero
                    + \hzero*cos(\p)*\Ny
                    - \hzero*sin(\p)*\Ty
                }
            )
            -- cycle;

        \draw[thick]
            plot[
                domain=\tmin:\tmax,
                samples=500,
                variable=\t
            ]
            (
                {
                    \a*\t*cos(deg(\t))
                    +(0.48-0.10/(\t+1))
                    *(-sin(deg(\t))-\t*cos(deg(\t)))
                    /sqrt(1+\t*\t)
                },
                {
                    \a*\t*sin(deg(\t))
                    +(0.48-0.10/(\t+1))
                    *( cos(deg(\t))-\t*sin(deg(\t)))
                    /sqrt(1+\t*\t)
                }
            );

        \draw[thick]
            plot[
                domain=\tmin:\tmax,
                samples=500,
                variable=\t
            ]
            (
                {
                    \a*\t*cos(deg(\t))
                    -(0.48-0.10/(\t+1))
                    *(-sin(deg(\t))-\t*cos(deg(\t)))
                    /sqrt(1+\t*\t)
                },
                {
                    \a*\t*sin(deg(\t))
                    -(0.48-0.10/(\t+1))
                    *( cos(deg(\t))-\t*sin(deg(\t)))
                    /sqrt(1+\t*\t)
                }
            );

        \draw[thick]
            plot[
                domain=180:0,
                samples=100,
                variable=\p
            ]
            (
                {
                    \xzero
                    + \hzero*cos(\p)*\Nx
                    - \hzero*sin(\p)*\Tx
                },
                {
                    \yzero
                    + \hzero*cos(\p)*\Ny
                    - \hzero*sin(\p)*\Ty
                }
            );

        \draw[dashed]
            plot[
                domain=\tmin:\tmax,
                samples=500,
                variable=\t
            ]
            (
                {\a*\t*cos(deg(\t))},
                {\a*\t*sin(deg(\t))}
            );

    \end{scope}

    \draw[->] (-\Xmax-0.5,0) -- (\Xmax+0.7,0) node[right] {$x$};
    \draw[->] (0,-\Xmax-0.5) -- (0,\Xmax+0.7) node[above] {$y$};

\end{tikzpicture}
        \caption{An unbounded spiral band converging to finite width.}
        \label{fig:spiral-band}
    \end{subfigure}
    \hfill
    \begin{subfigure}[t]{0.48\textwidth}
        \centering
        \begin{tikzpicture}[scale=0.6]

            \fill[gray!30]
                (-3.5,-3.5) rectangle (3.5,3.5);

            \foreach \x in {-3,-2,-1,0,1,2,3} {
                \foreach \y in {-3,-2,-1,0,1,2,3} {
                    \filldraw[
                        fill=white,
                        draw=black,
                        thick
                    ]
                    ({\x-0.25},{\y-0.25})
                    rectangle
                    ({\x+0.25},{\y+0.25});
                }
            }

        \end{tikzpicture}

        \caption{A periodic unbounded domain with square holes.}
        \label{fig:square-holes}
    \end{subfigure}

    \caption{Two examples of unbounded domains.}
    \label{fig:unbounded-domains1}
\end{figure}
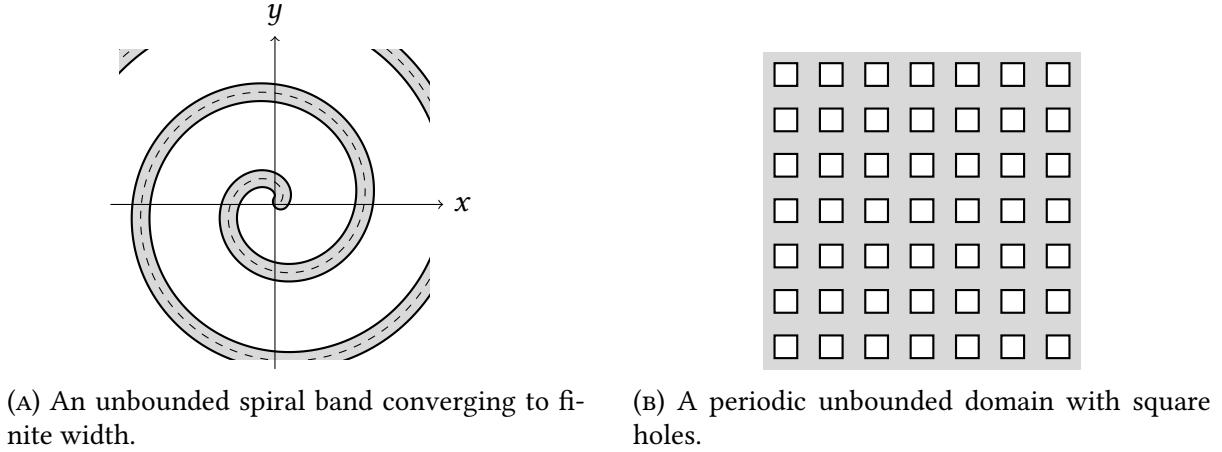

%


\subsection*{Acknowledgements}

MK is grateful to Kamil Kaleta for inspiring discussions related to this work. MK also thanks the anonymous piano players at Chicago O'Hare International Airport, where the core idea for the proof of Theorem 1.1 was conceived.

%
%

\bibliographystyle{plain}	
\bibliography{MainRef}

%
%

\end{document}